\documentclass[12pt]{article}
\usepackage{color}
\usepackage{amsmath}
\usepackage{amsfonts}
\usepackage{amssymb}
\usepackage{amsthm}
\usepackage{array}
\usepackage{geometry}
\usepackage{algpseudocode}
\usepackage{lscape}
\usepackage{hyperref}
\hypersetup{pdfborder=0 0 0}
\usepackage{enumerate}
\ifdefined\directlua
\usepackage{polyglossia}
\setmainlanguage{english}

\usepackage{fontspec}
\usepackage{microtype}

\usepackage{pdftexcmds}
\makeatletter
\ifcase\pdf@shellescape
\relax\or
\usepackage{luacode}
\begin{luacode}
  function prtgit()
  local cmd="git show -s --format='
  print(cmd)
  local r=io.popen(cmd):read("*a")
  if (r) then
  tex.print([[\string\def\string\COMMIT{]]..r..[[}]])
  end
  end
\end{luacode}
\directlua{prtgit()}
\or
\relax\fi
\makeatother
\ifdefined\COMMIT
        \usepackage{background}
        \backgroundsetup{%
         pages=all, placement=bottom,angle=0,scale=1.6,%
         vshift=20pt,hshift=0pt,
         contents={Commit version:\COMMIT}}
\fi
\fi

\newenvironment{ppmatrix}{\setlength\arraycolsep{0.1pt}\begin{pmatrix}}{\end{pmatrix}}

\newcommand{\FF}{\mathbb F}

\newcommand{\KK}{\mathbb K}

\newcommand{\cG}{\mathcal G}

\newcommand{\GG}{\mathbb G}
\newcommand{\Tr}{\mathrm{Tr}}

\newcommand{\HH}{\mathbb H}
\newcommand{\cC}{\mathcal C}

\newcommand{\fmm}{\mathfrak w}

\newcommand{\cH}{\mathcal H}

\newcommand{\PG}{\mathrm{PG}}

\newcommand{\fA}{\mathfrak A}

\newcommand{\cod}{\operatorname{codim}}

\newtheorem{theorem}{Theorem}
\newtheorem{lemma}[theorem]{Lemma}
\newtheorem{corollary}[theorem]{Corollary}

\theoremstyle{definition}

\author{Ilaria Cardinali and Luca Giuzzi}
\title{Implementing Line-Hermitian Grassmann codes}

\begin{document}
\maketitle
\begin{abstract}
  In~\cite{IL17c} we introduced line Hermitian Grassmann codes
  and determined their parameters.
  The aim of this paper is to present (in the spirit of \cite{IL17}) an algorithm for the point enumerator of a line Hermitian Grassmannian which can be usefully applied to get efficient encoders, decoders and error correction algorithms for the aforementioned codes.

  \end{abstract}

\section{Introduction}\label{sec1}
Let $V$ be a vector space of dimension $K$ over a (finite) field
$\KK$ and suppose
$\Omega$ to be a {\it{projective system}} of $\PG(V)$ of order $N$, that is a set of $N$ distinct points of $\PG(V)$ such that $\dim\langle\Omega\rangle=\dim(V)$.
An \emph{enumerator} for $\Omega$ is a
bijection
$\iota:\{0,1,\ldots,N-1\}\to\Omega$ which can be easily computed and inverted.

Enumerators are interesting because of their relevance to applications, as they provide an efficient
way to represent and access the elements of a (ordered) list, without requiring a large amount of storage.

For example, when $\Omega$ is the pointset of a Grassmann variety, a (point) enumerator
efficiently constructs a subspace for any integer in the domain. This
plays a key role in the implementation of Grassmann codes, since it is possible to
determine the components of the codewords without having to explicitly construct their full
generator matrix,  a process which would be very expensive both
computationally and in terms of storage.

In the present paper we shall be concerned with point enumerators of line Hermitian Grassmannians. In this way, we continue a project started in~\cite{IL17} where we introduced point enumerators for line polar Grassmannians of orthogonal \cite{ILP14,IL16} and symplectic type
\cite{IL15}.

Before stating our main results, we shall recall the definition of Hermitian Grassmannians and set the notation in Section~\ref{sec1.1}; next, in Section~\ref{sec1.2}, we shall provide some basics about polar Grassmann codes. The organization of the paper and the main results are outlined in Section~\ref{organization of the paper}.

\subsection{Hermitian Grassmannians and their embeddings}\label{sec1.1}
Assume $\KK=\FF_{q^2}$ to be a finite field of order $q^2$
and let $V:=V(m, \KK)$ be a  $m$-dimensional vector space over $\KK$ and $k\in \{1,\dots, m-1\}$.
Let $\cG_{m,k}$  be the  $k$-Grass\-mann\-ian of the projective
space $\PG(V)$, that is the point--line
geometry whose points are the $k$-dimensional subspaces of $V$
and whose lines are the sets
\[ \ell_{W,T}:=\{ X: W\leq X\leq T, \dim X=k \} \]
with $\dim W=k-1$ and  $\dim T=k+1$.


Let
$e_{k}:\cG_{m,k}\to\PG(\bigwedge^kV)$ be the Pl\"ucker (or Grassmann) embedding of $\cG_{m,k}$, which maps
any arbitrary
$k$--dimensional subspace $X=\langle v_1,v_2,\ldots,v_k\rangle$ of
$V$ to the point $e_k(X):=[v_1\wedge v_2\wedge\cdots\wedge v_k]$ of $\PG(\bigwedge^kV)$.
Note that lines of $\cG_{m,k}$ are mapped onto (projective) lines of $\PG(\bigwedge^kV)$.
The dimension $\dim(e_k)$
of the embedding $e_k$ is defined as the vector dimension of the subspace spanned by its image.
It is well known that $\dim(e_k)={m\choose k}.$

The image $e_{k}(\cG_{m,k})$ of the Pl\"ucker embedding is a projective variety
of $\PG(\bigwedge^kV)$, called \emph{Grassmann variety} and denoted by $\GG(m,k)$.



Suppose now that $V$
is equipped with a non-degenerate Hermitian form $\eta$ of Witt index $n$ (hence
either $m=2n+1$ or $m=2n$).

The Hermitian $k$-Grassmannian  
induced by $\eta$ 
is defined for $k=1,\ldots,n$ as the geometry having as
points the totally $\eta$--isotropic
subspaces of $V$ of dimension $k$ and as lines
\begin{itemize}
\item for $k< n$, the sets of the form
   \[ \ell_{W,T}:=\{ X: W\leq X\leq T, \dim X=k \} \]
   with $T$ totally $\eta$--isotropic and $\dim W=k-1$, $\dim T=k+1$.
 \item for $k=n$, the sets of the form
   \[ \ell_{W}:=\{ X: W\leq X, \dim X=n, X\,\, {\text{totally $\eta$--isotropic}}\} \]
   with $\dim W=n-1$.
 \end{itemize}
 We will denote a Hermitian $k$-Grassmannian either by the symbol $\cH_{n,k}$ when we
 do not want to mention the parity of $m$ or by the symbols $\cH_{n,k}^{even}$ (for $m=2n$) respectively  $\cH_{n,k}^{odd}$ (for $m=2n+1$) when the parity of $m$ plays a significant role.

  Clearly, for $k=1,\ldots,n$, the point-set of $\cH_{n,k}$ is always
 a subset of that of $\cG_{m,k}$.
 If $k=1$, $\cH_{n,1}$ (also denoted $\cH_{n}$ or $\cH_m$) stands for a Hermitian polar space of rank $n$ and if $k=n$, $\cH_{n,n}$ is usually called \emph{Hermitian dual polar space of rank $n$}.
Let $\varepsilon_{n,k}:=e_{k}|_{\cH_{n,k}}$ be the restriction of the Pl\"ucker embedding $e_{k}$ of
$\cG_{m,k}$
to the Hermitian $k$-Grassmannian $\cH_{n,k}.$
The map $\varepsilon_{n,k}$ is an embedding of $\cH_{n,k}$ called
\emph{Pl\"ucker (or Grassmann) embedding} of $\cH_{n,k}$ in
$\PG(\bigwedge^kV)$; its dimension has been proved to be $\dim(\varepsilon_{n,k})={{\dim(V)}\choose k}$ for $\dim(V)$ even and $k$ arbitrary by Blok and Cooperstein~\cite{BC2012} and for $\dim(V)$ arbitrary and $k=2$ by Cardinali and Pasini~\cite{IP14}.
Consider now the following projective system of $\PG(\bigwedge^kV)$:
\begin{equation}\label{Hermitian variety}
 \HH_{n,k}:=\varepsilon_{n,k}(\cH_{n,k})=\{\varepsilon_{n,k}(X)\colon X {\text{ point of }} \cH_{n,k}\}\subset \PG(\bigwedge^k V).
\end{equation}

Note that if $k=2$ and $n>2$ then $\varepsilon_{n,2}$ maps lines of $\cH_{n,2}$ onto projective lines of $\PG(\bigwedge^2 V)$, independently of the parity of $\dim(V)$, i.e. the embedding is \emph{projective}. Otherwise, if $n=k=2$ and $m=\dim(V)=5$ then the lines of $\cH_{2,2}^{odd}$ are mapped onto Hermitian curves, while if $m=\dim(V)=4$ then lines of $\cH_{2,2}^{even}$ are mapped onto Baer sublines of $\PG(\bigwedge^2 V)$.
In the latter case $\HH_{2,2}^{even}\cong Q^-(5,q)$ is
contained in a proper subgeometry, defined over the subfield $\FF_q$, of $\PG(\bigwedge^2V)$.

\subsection{Line Hermitian Grassmann codes}\label{sec1.2}
Given a projective system $\Omega$  of $\PG(V)$ of order $N$, where
$\dim(V)=K$, we can construct a $[N,K]$-linear
code $\cC(\Omega)$ associated to $\Omega$  as
a code whose generator matrix is the $K\times N$-matrix whose columns are
vector representatives of the points of $\Omega$; see~\cite{tvn}.
There is a well-known relationship between the maximum number of
points of $\Omega$ lying in a hyperplane of $\PG(V)$
and the minimum Hamming distance $d_{\min}$ of $\cC(\Omega)$, namely
\[ d_{\min}=N-\max_{{\begin{subarray}{l}
                  \Pi\leq \PG(V)\\
                   \cod(\Pi)=1
                   \end{subarray}}}\left|\Pi\cap\Omega\right|. \]
The case in which $\Omega$ is the point-set of a
Grassmann variety $\GG(m,k)$ has been extensively studied; see e.g.~\cite{R1,R2,R3,N96,GPP2009,GK2013,KP13}.
In this case the associated codes $\cC(\Omega)$ are called \emph{Grassman codes}.

In a series of papers we have investigated codes arising when $\Omega$ is a proper subvariety of $\GG(m,k)$,
namely when $\Omega$
is the image under the Pl\"ucker embedding of a polar
line Grassmannian; see \cite{IL13,ILP14,IL16} for the orthogonal case and
\cite{IL15} for the symplectic case.
Following the same approach as of~\cite{IL13}, we defined in~\cite{IL17c}  \emph{Hermitian Grassmann codes} as those projective codes
arising from the Pl\"ucker embedding of a Hermitian  Grassmannian (see Equation~\eqref{Hermitian variety}) and we determined their minimum distance,
also  characterizing the words of minimum
weight.

In particular, for line Hermitian Grassmann codes, i.e. taking $k=2$, we proved in \cite{IL17c} the following.
\begin{theorem}\label{main theorem 1}
A line Hermitian Grassmann code defined by a non--degenerate
Hermitian form on a vector space $V(m,q^2)$ is a $[N,K,d_{\min}]$-linear code where
   \[ N=\frac{(q^m+(-1)^{m-1})(q^{m-1}-(-1)^{m-1})(q^{m-2}+(-1)^{m-3})(q^{m-3}-(-1)^{m-3})}{(q^{2}-1)^2(q^{2}+1)};\]
   \[K={m\choose 2};\]
  \[d_{\min}=\begin{cases}\displaystyle
    q^{4m-12}-q^{2m-6} & \text{ if $m=4,6$ .} \\
    q^{4m-12} & \text{ if $m\geq 8$ is even.} \\
    q^{4m-12}-q^{3m-9} & \text{ if $m$ is odd.}
  \end{cases}
\]
\end{theorem}

\subsection{Organization of the paper and Main Results}\label{organization of the paper}
In Section~\ref{sec2} we recall the notion of prefix enumeration and provide counting algorithms
for the points of $\cH_{n,2}$.
In Sections~\ref{sec3} and~\ref{sec4} we compute
the number of totally $\eta$-singular lines of $V$
spanned by  vectors with a prescribed prefix.
The complexity of the prefix enumerators is discussed in Section \ref{sec-compl}, where we prove our main result. \\

\noindent {\bf Main Theorem}{\it
\begin{itemize}
\item[(i)] The computational complexity for the number of
  points of a line  Hermitian
  Grassmannian of $V(m,q^2)$, whose representation
  begins with a prescribed prefix, is $O(m^2)$.
 \item[(ii)]   The computational complexity for a point enumerator of
  a line Hermitian Grassmannian of $V(m,q^2)$ is $O(q^4m^3)$.
 \end{itemize} }

\medskip

Section~\ref{sec6} is dedicated to applications of the scheme
introduced in Sections~\ref{sec3} and \ref{sec4} to line Hermitian Grassmann codes.
We also propose some encoding/decoding and
error correction strategies which act locally on the components of
the codewords.

 \section{Preliminaries}\label{sec2}

\subsection{Point enumerator: notation and basics}\label{sec2.1}
Let $\fA$ be an alphabet equipped with  a total order relation $\prec$ and let $\mathcal{O}\subseteq \fA^m$ be a set of $m$-uples
with entries in $\fA$. For any $\omega\in \mathcal{O}$ and $t\leq m$ a non-negative integer, we define the \emph{prefix of length $t$} or
\emph{$t$-prefix} of $\omega$ as the $t$-uple of the first $t$-entries of $\omega.$
If $\alpha=(\alpha_1,\dots, \alpha_t)\in \fA^t$ and $\beta=(\beta_1,\dots, \beta_s)\in \fA^s$, we shall write $\alpha|\beta$ to refer to the $(t+s)$-uple $(\alpha, \beta)=(\alpha_1,\dots, \alpha_t,\beta_1,\dots, \beta_s)\in \fA^{t+s}$ and say that $\alpha|\beta$ is the \emph{concatenation of $\alpha$ and $\beta$}.
If $t=0$ then $\alpha=\emptyset$ and we let $\emptyset|\beta=\beta$.
Accordingly, $\alpha$ is \emph{the $t$-prefix} of $\alpha|\beta$.

Given  $\alpha=(\alpha_i)_{i=1}^t\in \fA^t$, $1\leq t\leq m$,  define $\mathcal{O}^{\alpha}$ as the set
all the concatenations $\alpha|\beta\in{\mathcal O}$ where $\beta$ varies in $\fA^{m-t}$, i.e.
\[ \mathcal{O}^{\alpha}:=\{ \alpha|\beta\in\mathcal{O}\colon \beta\in\fA^{m-t}\}=\{ \omega\in\mathcal{O}\colon \omega_i=\alpha_i,
  \forall i=1,\ldots,t\}.\]
Put also
\[ \mathcal{O}^{\emptyset}:=\mathcal{O}\,\,\text{and}\,\, \fA^0=\emptyset. \]
Suppose the following function is given
\begin{equation}\label{psi}
\psi:\left\{ \begin{array}{l}
  \bigcup_{t=0}^{m}\fA^t\to\{0,\ldots,|\mathcal{O}|\} \\
  \\
  \alpha\to |\mathcal{O}^{\alpha}|
\end{array}\right.
\end{equation}
Clearly $\psi(\alpha)=0$ if and only if
there is no word in $\mathcal O$ with prefix $\alpha$
 and $\psi(\alpha)=|\mathcal{O}|$ if
and only if all words in $\mathcal O$ have $\alpha$ as prefix.

For any $\omega\in\mathcal{O}$ and $i\leq m$ write $\omega_{\leq i}:=(\omega_1,\ldots,
\omega_i)$ for the $i$-prefix of $\omega$ and let ${\mathbb I}=\{0,1,\ldots,|\mathcal{O}|-1\}$. Note that $\omega_{\leq 0}=\emptyset$ and, for $x\in \fA$, $\psi(\omega_{\leq 0}|x)=\psi(x)$.

The function $\iota$ defined by
\begin{equation} \label{iota}
  \iota:\begin{cases}\mathcal{O}\to{\mathbb I} \\
  \iota(\omega):=\displaystyle
  \sum_{j=1}^{m}\sum_{\begin{subarray}{c}
      x\in\fA; \\
      x\prec \omega_j.\\
    \end{subarray}}\psi(\omega_{\leq j-1}|x).
\end{cases}
\end{equation}
is an enumerator for the set $\mathcal{O}$ and given any two elements
$\omega,\omega'\in\mathcal{O}$ we have $\iota(\omega)<\iota(\omega')$ if
and only if $\omega$ precedes $\omega'$ in the lexicographic ordering
of $\fA^m$ induced by $\prec$. See~\cite{Cover} where such a function was first introduced and also~\cite{IL16} for the details of its injectivity.
The inverse of the function $\iota$ can be computed as described  in
Table~\ref{tt}; see~\cite[Theorem 3.1]{IL16} for the details of the proof.
\begin{table}[h]
\caption{Inverse of $\iota$}
\begin{algorithmic}
\Require $i\in{\mathbb I}$
\State $i_1\gets 1$
\State $\omega\gets $ []
\For {$k=1,\ldots,m$}
\State
$M\gets\{y\in\fA\colon
 \psi(\omega_{\leq k-1}|y)>0 \text{ and } \sum_{x\prec y}\psi(\omega_{\leq k-1}|x)\leq i_k
  \}$
\State
 $\omega\gets \omega|(\max M)$
\State $i_{k+1}\gets i_{k}-\sum_{x\prec \omega_k}\psi(\omega_{\leq k-1}|x)$;
\EndFor
\State \Return $\omega$
\end{algorithmic}
\label{tt}
\end{table}


\subsection{Notation}
In this section we shall apply to the case of line Hermitian Grassmannians, what we introduced in general in Section~\ref{sec2.1}.
Before defining the analogue of the function $\psi$ (see Equation~\eqref{psi}) we need to recall
a few more notions and explain what we mean by representation of a given line.

Let $V=V(m,q^2)$ be a vector space of dimension $m$ over $\FF_{q^2}$ and let $B$ be a given basis of $V$. We will always write the coordinates of vectors with respect to $B$.

Up to projectivities, there is exactly one class of non-degenerate Hermitian forms on $V.$
So, without loss of generality,  we can take for $m$ odd the form $\eta$ to be
\begin{equation}\label{eta}
\eta_m(X,Y):= x_1^qy_1+x_2^qy_3+x_3^qy_2+\cdots+x_{m-1}^qy_{m}+x_{m}^qy_{m-1}
\end{equation}
  where $X=(x_i)_{i=1}^m,$ and $Y=(y_i)_{i=1}^m.$
   Accordingly, the points $X$ of the Hermitian polar space $\cH_m$ defined by $\eta_m$ satisfy the following
  equation
  \[ x_1^{q+1}+\Tr(\sum_{i=1}^{(m-1)/2}x_{2i}^qx_{2i+1})=0, \]
  where $\Tr(x):=x+x^q$ is the trace function from $\FF_{q^2}$ to $\FF_q$.

  For $m$ even, we will regard as the vector space $V_m$
  as
  the $m$-dimensional vector subspace of $V_{m+1}:=V(m+1,q^2)$ of equation $x_1=0$ and the Hermitian form defined on $V_m$ will be the restriction of the form $\eta_{m+1}$ defined above for the case of odd dimension
  to $V_m\times V_m$, i.e.   $\eta_m(X,Y):=\eta_{m+1}|_{V_m\times V_m}(X,Y)$.
Accordingly, if $m$ is even, the polar space $\cH_m$ has as points and lines, the points and lines of $\cH_{m+1}$ which are fully contained in the hyperplane of equation $x_1=0$. Hence, $\cH_n^{even}=\cH_n^{odd}\cap V_{m}$.

We shall denote by $\mu_m$, respectively $N_m$, the number of points, respectively the number of lines, of $\cH_m\subseteq \PG(m-1,q^2)$ (independently from the parity of $m$). It is well known that
\begin{equation}\label{pts and lines}
\mu_m:=\frac{(q^m+(-1)^{m-1})(q^{m-1}- (-1)^{m-1})}{(q^2-1)}\,\, \textrm{and}\,\, N_m:=\frac{\mu_m\mu_{m-2}}{q^2+1}.
\end{equation}

Recall that a $(2\times t)$-matrix $G$ is
said to be in \emph{Hermite normal form} or in \emph{row reduced
echelon form} (RREF, in brief) if it is in row-echelon form, the leading
non-zero entry of each row is $1$ and all entries above a leading entry
are $0$.

The points of $\cH_m$ will be represented
as vectors normalized on the left, i.e. whose first non-zero entry is $1$.
For them the \emph{alphabet} consists of the elements of $\FF_{q^2}$ where
$\prec$ is an arbitrary total order relation defined on it
satisfying the condition:  $\forall y\in\FF_{q^2}\setminus\{0\}: 0\prec y$.

The elements of a line Hermitian Grassmannian, i.e. the points of $\cH_{n,2}$,
are the totally singular lines $\ell\subseteq\PG(m-1,q^2)$ and each of them admits a unique representation
in \emph{row-reduced echelon form} (RREF), i.e. for each line $\ell$  of $\PG(V)$ regarded as a point of $\cH_{n,2}$,
there are two uniquely determined vectors $X,Y\in\FF_{q^2}^{m}$ such that $\ell=\langle X,Y\rangle$
and $G_{\ell}:=\begin{ppmatrix} X\\Y\end{ppmatrix}$ is a
$2\times m$-matrix in RREF. We call $G_{\ell}$ the \emph{representation} of $\ell$.

With a slight abuse of notation we shall say that a point (resp. a line) has prefix $\alpha$ if its
representation with respect to the basis $B$ has prefix $\alpha$. Clearly, when $m$ is even
we can identify the points of the polar space $\cH_{m}$ with those of $\cH_{m+1}$ having prefix $0$, while
the lines of $\cH_{m}$ correspond to the lines of $\cH_{m+1}$ with prefix $\begin{ppmatrix} 0\\0\end{ppmatrix}$.

Recall that, in the case of lines, the alphabet $\fA$  consists of all
column vectors in $\FF_{q^2}^2$ of the form $\begin{ppmatrix} \alpha \\ \beta \end{ppmatrix}$
with $\alpha,\beta\in\FF_{q^2}$. With a slight abuse of notation, we shall denote the order induced lexicographically on $\FF_{q^2}^2$ by
$\prec$ with the same symbol $\prec$, so that
\[ \begin{ppmatrix}\alpha\\ \beta\end{ppmatrix}\prec\begin{ppmatrix}\gamma\\ \delta\end{ppmatrix}
  \Leftrightarrow
  (\alpha\prec\gamma) \mbox{ or } ( (\alpha=\gamma) \text{ and } \beta\prec\delta). \]

Under the assumptions we made above, since $\cH_n^{even}=\cH_n^{odd}\cap V_{m}$, for all points $p\in\cH_n^{even}$
and $p'\in\cH_n^{odd}\setminus\cH_n^{even}$ we have $\iota(p)<\iota(p')$.
Likewise, for any two lines $\ell\in\cH_{n,2}^{even}$ and $\ell'\in\cH_{n,2}^{odd}\setminus\cH_{n,2}^{even}$ (where $\ell=\ell'\cap V_m$)
we also have $\iota(\ell)<\iota(\ell')$.
This implies that, we can define point and line enumerators for
the polar spaces $\cH_{n}^{even}$ as the restriction of the corresponding point and line
enumerators for the polar spaces $\cH_n^{odd}$.
So, in the remaining part of the paper, we shall only explicitly deal with enumerators for
$\cH_n^{odd}$, i.e. assume $m$ to be odd.

\section{Point enumerators}\label{sec3}
 For any $t$-uple $D_t=(d_i)_{i=1}^t$, $t\leq m$, denote by $\theta(D_t)$ the number of points
  of $\cH_m$ having $D_t$ as prefix, i.e.
  \[\theta(D_t):=|\{[D_t|X_{m-t}]\colon [D_t|X_{m-t}] \textrm{ is a point of }\cH_m \}|.\]
In this section we will explain how to explicitly compute the point counting function $\theta.$

We suppose $m$ is odd.
\begin{itemize}
\item if $D_t$ is not normalized on the left, then
  return $0$;
  \item if $t$ is odd, there are two possible subcases:
    \begin{enumerate}
    \item \framebox{$D_t=\mathbf{0}$}. Then the number of points whose representation begins with $D_t$ is
      exactly the same as the number of points of a Hermitian
      variety $\cH_{m-t}$; so return
      \[ \boxed{\theta(D_t)=\mu_{m-t}}. \]
    \item \framebox{$D_t\neq\mathbf{0}$}.
	    Let $s=(q^2-1)\mu_{m-t}+1$ be the number
      of vectors contained in a non-degenerate Hermitian variety in $\PG(V_{m-t})$ and
      \begin{enumerate}
      \item if $\eta_{t}(D_t,D_t)=0$, then return $s$;
      \item if $\eta_{t}(D_t,D_t)\neq 0$, then let $X\in V(m-t,q^2)$ such that $\eta_{m-t}(X,X)\not=0.$ Then there are $q+1$ values of $\lambda$ such that $\eta_{m-t}(\lambda X,\lambda X)=-\eta_{t}(D_t,D_t)$ since the equation $\lambda^{q+1}=\frac{-\eta_t(D_t,D_t)}{\eta_{m-t}(X,X)}$ has exactly $q+1$ solutions. In other words, for each point $X$ not in $\cH_{m-t}$ (by convention the points of $\cH_{m-t}$ are represented by vectors normalized on the left) there are $q+1$ vectors $Y\in V(m-t,q^2)$ such that $D_t|Y$ is a point of $\cH_m.$
 So, return
        $(q^{2(m-t)}-s)/(q-1)$.
      \end{enumerate}
In summary,
\[ \boxed{\theta(D_t)=\begin{cases}
		 s & \mbox{ if } \eta_t(D_t,D_t)=0 \\
		 (q^{2(m-t)}-s)/(q-1) & \mbox{ if } \eta_t(D_t,D_t)\neq0
 \end{cases}}. \]
    \end{enumerate}
  \item if $t$ is even, there are three possible subcases:
    \begin{enumerate}
    \item \framebox{$D_t=\mathbf{0}$}. In this case we just return the number of points
      whose  $t+1$ entry is $0$ plus the number of points whose $t+1$ entry is
      $1$ and use the arguments of the previous cases; so
      \[ \boxed{\theta(D_t)=\theta(D_t|0)+\theta(D_t|1)}. \]
    \item \framebox{$D_t\neq\mathbf{0}$ and $d_t=0$}. In this case we need to compute the number
      of vectors whose representation begins with $D_t$;
      observe that this number does not depend
      on the value of the  entry in position $t+1$, since we have
      \[ d_1^{q+1}+\Tr(d_2^qd_3+\cdots+ 0^qx_{t+1}+\cdots)=0 \]
      so, this is $q^2$ times the number of points with odd prefix $D_t|0$.
      So,
      \[ \boxed{\theta(D_t)=q^2\theta(D_t|0)}. \]
    \item
      \label{evDtn0}
      \framebox{$D_t\neq\mathbf{0}$ and $d_t\not=0$}. In this case we have to count the number of vector solutions of
      \[ d_1^{q+1}+\Tr(d_2^qd_3+\cdots d_{t-2}^qd_{t-1})+\Tr(d_t^qx_{t+1})+
        \Tr(x_{t+2}^qx_{t+3}+\cdots+x_{m-1}^qx_{m})=0. \]
      For any choice of $x_{t+2},\ldots,x_{m}$ there are $q$ possibilities
      for $x_{t+1}$ such that
      \[ \Tr(d_t^qx_{t+1})=-d_1^{q+1}-\Tr(d_2^qd_3+\cdots d_{t-2}^qd_{t-1})
        -\Tr(x_{t+2}^qx_{t+3}+\cdots+x_{m-1}^qx_{m}).\]
      So, in this case, we return
      \[ \boxed{\theta(D_t)=q\cdot q^{2(m-t-1)}}. \]
    \end{enumerate}
  \end{itemize}
  It is straightforward to see that to compute the function $\theta(D_t)$ where
  $D_t$ is a prefix of length $t\leq m$ requires at worst to evaluate
  $\eta_t(D_t,D_t)$, that
  is to perform $t$ products and $t$ conjugations. So,
  \begin{lemma}   \label{complexity theta}
  The computational complexity
  of the function $\theta$ is  $O(t)\leq O(m)$.
  \end{lemma}

\section{Line Enumerators}\label{sec4}
Let $\ell$ be a line of $\cH_m$ with prefix $D_t$ of length $t$ ($\leq m$). More explicitly, put $D_t:=\begin{ppmatrix} A_t\\ B_t\end{ppmatrix}$ where $A_t:=(a_i)_{i=1}^t$, $B_t:=(b_i)_{i=1}^t$,
and define $\hat{A}:=(A_t, x_{t+1},\dots, x_m)$ and $\hat{B}:=(B_t, y_{t+1},\dots, y_m)$ so that $\ell=\langle  \hat{A},\hat{B}\rangle.$ Suppose $D_t$ is in RREF form and denote by $\psi(D_t)$ the number of lines in $\cH_m$ whose
  representation in RREF begins with $D_t$. Define  $\psi^E(D_t)$ and $\psi^O(D_t)$ as follows:
  \[ \psi(D_t)=:\begin{cases}
      \psi^E(D_t) & \text{ if $t$ is even } \\
      \psi^O(D_t) & \text{ if $t$ is odd.}
      \end{cases} \]
In this section we will compute $\psi(D_t)$ with some recursive formulas. To this aim, we need to determine the number of solutions of the
    following system of equations:
    \begin{equation}\label{system}
    \begin{cases}
      \eta_m(\hat{A},\hat{A})=0\\
      \eta_m(\hat{B},\hat{B})=0\\
      \eta_m(\hat{A},\hat{B})=0\\
     \end{cases}.
   \end{equation}
We shall distinguish two cases, depending on the parity of $t$.
\subsection{Even prefix $t$}\label{t even}
It $t=0$ then $\psi(D_0)=\psi^E(\emptyset)$ is clearly the number of lines  of $\cH_m$.
If $t>0$ system~\eqref{system} can be written as
  \begin{equation}
    \label{teven}
    \begin{cases}
      a_1^{q+1}+\Tr(a_2^qa_3+\cdots+a_{t-2}^qa_{t-1})+\Tr(a_t^qx_{t+1})+\\
      \qquad\qquad\qquad\qquad\qquad\qquad+\Tr(x_{t+2}^qx_{t+3}+\cdots+
      x_{m-1}^qx_{m})=0 \\
      b_1^{q+1}+\Tr(b_2^qb_3+\cdots+b_{t-2}^qb_{t-1})+\Tr(b_t^qy_{t+1})+ \\
      \qquad\qquad\qquad\qquad\qquad\qquad
      +\Tr(y_{t+2}^qy_{t+3}+\cdots+
      y_{m-1}^qy_{m})=0 \\
      a_1^qb_1+a_2^qb_3+a_3^qb_2+ \cdots +a_{t-2}^qb_{t-1}+a_{t-1}^qb_{t-2} +a_t^qy_{t+1}+ x_{t+1}^qb_{t}+\\
      \qquad\qquad x_{t+2}^qy_{t+3}+x_{t+3}^qy_{t+2}
      +\cdots+ x_{m-1}^qy_{m}+x_{m}^qy_{m-1}=0
    \end{cases}
  \end{equation}
or, equivalently, as

  \begin{equation}
    \label{tevenp}
    \begin{cases}
      \eta_m(A_t|0^{m-t},A_t|0^{m-t})+\Tr(a_t^qx_{t+1})+\Tr(x_{t+2}^qx_{t+3}+\cdots+
      x_{m-1}^qx_{m})=0 \\
      \eta_m(B_t|0^{m-t},B_t|0^{m-t})+\Tr(b_t^qy_{t+1})+\Tr(y_{t+2}^qy_{t+3}+\cdots+
      y_{m-1}^qy_{m})=0 \\
      \eta_m(A_t|0^{m-t},B_t|0^{m-t})+a_t^qy_{t+1}+x_{t+1}^qb_{t}+x_{t+2}^qy_{t+3}+x_{t+3}^qy_{t+2}+\\
      \qquad\qquad\qquad\qquad\qquad\qquad\qquad\qquad
      +\cdots+ x_{m-1}^qy_{m}+x_{m}^qy_{m-1}=0.
    \end{cases}
  \end{equation}

We can always suppose that either $a_t=0$ or $b_t=0$;
    otherwise, if for example $b_t\not=0,$ replace $D_t$ by the prefix
    $\begin{ppmatrix} A_t-\frac{a_t}{b_t}B_t \\ B_t \end{ppmatrix}$ in~\eqref{tevenp}.
    The following cases need to be considered.
    \begin{enumerate}[(E1)]
    \item\label{E0} \framebox{$t=0$}. All lines in $\cH_m$
      have prefix $\emptyset$; thus,
      \[ \boxed{\psi^E(\emptyset)=N_m}. \]
    \item\label{E1} \framebox{$A_t=\mathbf{0}=B_t$}.
      The lines with prefix $D_t$
      correspond to the lines contained in the degenerate Hermitian variety embedded in a vector space of dimension $m-t$
      satisfying the following equations
            \[ (x_{t}^qx_{t+1}+x_{t+1}^qx_{t}+x_{t+2}^qx_{t+3}+x_{t+3}^qx_{t+2}+\cdots+x_{m-1}^qx_{m}=0) \wedge (x_t=0). \]
      Hence
      \[ \boxed{\psi^E(D_t):=\mu_{m-t-1}+q^4N_{m-t-1}}. \]
    \item\label{E2} \framebox{$a_t\neq 0$ and $B_t=\mathbf{0}$}.
      For any choice of $X_{m-t}$ such that $\eta_m(A_t|X_{m-t},A_t|X_{m-t})=0$,
      there are $\mu_{m-t-1}$ points $Y_{m-t}$ such that
      $\langle A_t|X_{m-t}, 0_t|Y_{m-t}\rangle$ is a line of $\cH_m$.
      Since
      $\begin{pmatrix} A_t & X_{m-t} \\ 0_t & Y_{m-t}\end{pmatrix}$ and
      $\begin{pmatrix} A_t & X_{m-t}-\alpha Y_{m-t} \\ 0_t & Y_{m-t} \end{pmatrix}$
      represent the same line for all $\alpha\in\FF_{q^2}$,
      any line is represented $q^2$ distinct times.
      So, we get
      \[ \psi^E(D_t)=\frac{1}{q^2}\mu_{m-t-1}\theta(A_t). \]
      Using the value provided by point~\ref{evDtn0} of case even in Section~\ref{sec3} for $\theta(A_t)$,
      this gives
      \[ \boxed{\psi^E(D_t):=q^{2m-2t-3}\mu_{m-t-1}}. \]

    \item\label{E3} \framebox{$a_t=0$ and  $b_t\not= 0$}. For any fixed vector $Y=(B_t|y_{t+1}|Y')$ with prefix $B_t$ such that  $\eta_m(B_t|y_{t+1}|Y',B_t|y_{t+1}|Y')=0$
       and any choice of $X'=(x_{t+2},\ldots,x_{m})$ such that the vector
      $(A_t|0|X')$ satisfies $\eta_m(A_t|0|X',A_t|0|X')=0$,
      the third equation of System~\eqref{tevenp} uniquely determines $x_{t+1}$. Hence for any value of $(x_{t+2},\cdots, x_m)$ and $(y_{t+1},\cdots, y_m)$
 satisfying the first and the second equation of System~\eqref{tevenp}, the third equation provides only one solution of $x_{t+1}$; so the number of solutions of the system can be computed by only considering the first two equations.
 The number of solutions of the second equation is $\theta(B_t)$, where we recall from Section~\ref{sec3} that $\theta(B_t)$ is the number of points of $\cH_m$ having $B_t$ as prefix. As for the first equation, since $a_t=0$, we have
      $\theta(A_t|0)=\theta(A_t|\alpha)$ for any $\alpha\in\FF_{q^2}$. Hence, the number of solutions of the first equation is $\theta(A_t|0).$
        Since
\[\theta(A_t)=\sum_{\alpha\in\FF_{q^2}}\theta(A_t|\alpha)=\sum_{\alpha\in\FF_{q^2}}\theta(A_t|0)=q^2\theta(A_t|0), \]
      we have $\theta(A_t|0)=\frac{1}{q^2}\theta(A_t)$.
      It follows that $\psi^E(D_t)=\theta(A_t|0)\theta(B_t)$; hence
      \[ \boxed{\psi^E(D_t):=\frac{1}{q^2}\theta(A_t)\theta(B_t)}. \]
    \item\label{E4} \framebox{$a_t\not=0$ and $B_t\neq\mathbf{0}$}.
      Since $a_t\not= 0$ we necessarily have $b_t=0$. This case is analogous to the previous one so
        \[ \boxed{\psi^E(D_t)=\frac{1}{q^2}\theta(A_t)\theta(B_t)}. \]
    \item\label{E5} \framebox{$a_t=b_t=0$ and $A_t\neq\mathbf{0}\neq B_t$}.
      By System~\eqref{teven}, the set of lines of $\cH_m$ with (even) prefix $D_t=\begin{ppmatrix}A_t\\
      B_t\end{ppmatrix}$ is the same as the disjoint union of the sets of lines with (odd) prefix $D_{t+1}'=\begin{ppmatrix}A_t&|&\alpha\\
      B_t&|&\beta\end{ppmatrix}$ as $\alpha$ and $\beta$ vary in $\FF_{q^2}.$ Hence
      \[ \psi^E(D_t)=\sum_{\alpha,\beta\in\FF_{q^2}}\psi^O\left(
          \begin{ppmatrix} A_t&|&\alpha \\ B_t&|&\beta \end{ppmatrix}\right). \]
      Observe that
      \[ \psi^O\left(
          \begin{ppmatrix} A_t&|&\alpha \\ B_t&|&\beta \end{ppmatrix}\right)=
        \psi^O\left(
          \begin{ppmatrix} A_t&|&0 \\ B_t&|&0 \end{ppmatrix}\right) \]
      for any $\alpha,\beta\in\FF_{q^2}$. So,
      \[ \boxed{\psi^E(D_t):=q^4\psi^O\left(  \begin{ppmatrix} A_t&|&0 \\ B_t&|&0 \end{ppmatrix}\right)}. \]
      The function $\psi^O$ shall be analyzed in the next section.
    \item\label{E6} \framebox{$a_t=b_t=0$ and $A_t\not=\mathbf{0}=B_t$}.
      In this case $\eta_{t+1}(A_t|\alpha,A_t|\alpha)=\eta_{t+1}(A_t|0,A_t|0)$
      and $\eta_{t+1}(A_t|\alpha,\mathbf{0}|0)=0$ for all $\alpha\in\FF_{q^2}$.
      So, in particular,
     \[ \psi^O\left(
          \begin{ppmatrix} A_t&|&\alpha \\ \mathbf{0}&|&0 \end{ppmatrix}\right)=
        \psi^O\left(
          \begin{ppmatrix} A_t&|&0 \\ \mathbf{0}&|&0 \end{ppmatrix}\right) \]
      for any $\alpha\in\FF_{q^2}$. Thus
      \[ \boxed{ \psi^E(D_t):=}
        \sum_{\alpha\in\FF_{q^2}}\psi^O\left(\begin{ppmatrix}A_t&|&\alpha \\
            \mathbf{0}&|&0
            \end{ppmatrix}\right)+\psi^O\left(\begin{ppmatrix}A_t&|&0\\ \mathbf{0}&|&1
            \end{ppmatrix}\right)=\]
        \[=\boxed{q^2\psi^O\left(\begin{ppmatrix}A_t&|&0 \\\mathbf{0}&|&0
            \end{ppmatrix}\right)+\psi^O\left(\begin{ppmatrix}A_t&|&0\\ \mathbf{0}&|&1
            \end{ppmatrix}\right)}. \]
    \end{enumerate}
    \subsection{Odd prefix $t$}\label{t odd}
    In this section we shall examine the case in which $t$ is odd.
    System~\eqref{system} can be written as
    \begin{equation} \label{todd}
    \begin{cases}
      a_1^{q+1}+\Tr(a_2^qa_3+\cdots+a_{t-1}^qa_{t}) +\Tr(x_{t+1}^qx_{t+2}+\cdots+
      x_{m-1}^qx_{m})=0 \\
      b_1^{q+1}+\Tr(b_2^qb_3+\cdots+b_{t-1}^qb_{t})+\Tr(y_{t+1}^qy_{t+2}+\cdots+
      y_{m-1}^qy_{m})=0 \\
      a_1^qb_1+a_2^qb_3+a_3^qb_2+\cdots+a_{t-1}^qb_{t}+a_{t}^qb_{t-1}+\\
      \qquad\qquad\qquad +x_{t+1}^qy_{t+2}+x_{t+2}^qy_{t+1}
      +\cdots+ x_{m-1}^qy_{m}+x_{m}^qy_{m-1}=0
    \end{cases}
  \end{equation}
 or, equivalently, as
  \begin{equation}
    \label{toddp}
    \begin{cases}
      \eta_t(A_t,A_t)+\Tr(x_{t+1}^qx_{t+2}+\cdots+
      x_{m-1}^qx_{m})=0 \\
      \eta_t(B_t,B_t)+\Tr(y_{t+1}^qy_{t+2}+\cdots+
      y_{m-1}^qy_{m})=0 \\
      \eta_t(A_t,B_t)+x_{t+1}^qy_{t+2}+x_{t+2}^qy_{t+1}
      +\cdots+ x_{m-1}^qy_{m}+x_{m}^qy_{m-1}=0.
    \end{cases}
  \end{equation}
The following cases need to be considered.
   \begin{enumerate}[(O1)]
   \item\label{O1} \framebox{$t=m$} Then either $D_m=\begin{ppmatrix} A_m\\ B_m\end{ppmatrix}$ represents a line $\ell=\langle A_m,B_m\rangle$ of $\cH_m$
        or not; so we have
        \[ \boxed{\psi^O(D_m):=\begin{cases}
            1 & \text{if } \eta_m(A_m,B_m)=\eta_m(A_m,A_m)=\eta_m(B_m,B_m)=0 \\
            0 & \text{otherwise}
            \end{cases}}. \]
      \item\label{O2} \framebox{$A_t=\mathbf{0}=B_t$}. The lines with prefix $D_t$ correspond to the lines contained in a non-degenerate Hermitian variety
        embedded in a vector space of dimension $m-t$ satisfying the following equation
            \[x_{t+1}^qx_{t+2}+x_{t+2}^qx_{t+1}+\cdots+x_{m-1}^qx_{m}=0.\] Hence
        \[ \boxed{\psi^O(D_t):=N_{m-t}}. \]
        \item\label{O3} \framebox{$A_t\not=\mathbf{0}$ and $B_t=\mathbf{0}$}
Then
      \begin{multline*} \psi^O(D_t):=\sum_{\alpha\neq0}\psi^E\left(\begin{ppmatrix}A_t&|&\alpha\\
              \mathbf{0}&|&0\end{ppmatrix}\right)+
          \psi^E\left(\begin{ppmatrix}A_t&|&0\\
                \mathbf{0}&|&1\end{ppmatrix}\right)+ \\
          +\sum_{\alpha}\psi^O\left(\begin{ppmatrix}A_t&|&0&|&\alpha\\
                B_t&|&0&|&0\end{ppmatrix}\right)+
            \psi^O\left(\begin{ppmatrix}A_t&|&0&|&0\\
                B_t&|&0&|&1\end{ppmatrix}\right).
            \end{multline*}
            \begin{enumerate}
          \item Observe that
            \[ \psi^E\left(\begin{ppmatrix}A_t&|&\alpha\\
                  \mathbf{0}&|&0\end{ppmatrix}\right)=
              \psi^E\left(\begin{ppmatrix}A_t&|&1\\
                  \mathbf{0}&|&0\end{ppmatrix}\right) \]
            for any $\alpha\neq 0$; so, the first contribution is
            \[\sum_{\alpha\neq0}\psi^E\left(\begin{ppmatrix}A_t&|&\alpha\\
                  \mathbf{0}&|&0\end{ppmatrix}\right)=(q^2-1)
              \psi^E\left(\begin{ppmatrix}A_t&|&1\\
                  \mathbf{0}&|&0\end{ppmatrix}\right) \]
          \item Also
            \[ \psi^O\left(\begin{ppmatrix}A_t&|&0&|&\alpha\\
                  \mathbf{0}&|&0&|&0\end{ppmatrix}\right)=
              \psi^O\left(\begin{ppmatrix}A_t&|&0&|&0\\
                  \mathbf{0}&|&0&|&0\end{ppmatrix}\right) \]
            for any $\alpha$; so the third contribution is
            \[ \sum_{\alpha}\psi^O\left(\begin{ppmatrix}
                  A_t&|&0&|&\alpha \\ B_t&|&0&|&0\end{ppmatrix}\right)=q^2\psi^O\left(\begin{ppmatrix}A_t&|&0&|&0\\
                  \mathbf{0}&|&0&|&0\end{ppmatrix}\right). \]
          \end{enumerate}
          It follows that
          \[ \boxed{\begin{array}{ll}\psi^O(D_t):=&(q^2-1)
              \psi^E\left(\begin{ppmatrix}A_t&|&1\\
                  \mathbf{0}&|&0\end{ppmatrix}\right)+ \psi^E\left(\begin{ppmatrix}A_t&|&0\\
                \mathbf{0}&|&1\end{ppmatrix}\right)+ \\&
              \qquad+q^2\psi^O\left(\begin{ppmatrix}A_t&|&0&|&0\\
                  \mathbf{0}&|&0&|&0\end{ppmatrix}\right)
                +\psi^O\left(\begin{ppmatrix}A_t&|&0&|&0\\
                    B_t&|&0&|&1\end{ppmatrix}\right).
                          \end{array}}\]

      \item\label{O4} \framebox{$B_t\neq\mathbf{0}$}  Then
        \begin{multline*} \psi^O(D_t):=\sum_{\alpha\neq0\neq\beta}\psi^E\left(\begin{ppmatrix}A_t&|&\alpha\\
              B_t&|&\beta\end{ppmatrix}\right)+
          \sum_{\alpha\neq0}\psi^E\left(\begin{ppmatrix}A_t&|&\alpha\\
                B_t&|&0\end{ppmatrix}\right)+\\
         +\sum_{\beta\neq0}\psi^E\left(\begin{ppmatrix}A_t&|&0\\
                B_t&|&\beta\end{ppmatrix}\right)+
            \psi^E\left(\begin{ppmatrix}A_t&|&0\\
                B_t&|&0\end{ppmatrix}\right).
          \end{multline*}
          In order to determine the various contributions,
          for $A_t,B_t\in V(t,q^2)$, define
          \[ \zeta(A_t,B_t):=\sum_{\alpha\neq0\neq\beta}\psi^E\left(\begin{ppmatrix}A_t&|&\alpha\\
                  B_t&|&\beta\end{ppmatrix}\right). \]
            \begin{lemma}
              \label{lemma3}
            \[ \zeta(A_t,B_t)=(q^2-1)\theta(B_t|1)[\xi(A_t,B_t)s+
              (q^2-1-\xi(A_t,B_t))\frac{q^{2m-2t-4}-s}{q-1}] \]
            where $\xi(A_t,B_t)$ is the number of values $\lambda\in\FF_{q^2}$
            with $\lambda\neq 0$ such that
            \[ \eta_m(A_t-\lambda B_t|0^{m-t},A_t-\lambda B_t|0^{m-t})=0 \]
            and $s=(q^2-1)\mu_{m-t-2}+1$.
          \end{lemma}

          \begin{proof}
Since
\[ \psi^E\left(\begin{ppmatrix} A_t &|& \alpha \\ B_t &|&\beta\end{ppmatrix}\right)=
  \psi^E\left(\begin{ppmatrix} A_t-\lambda_{\alpha,\beta} B_t &|& \alpha-\lambda_{\alpha,\beta}\beta \\ B_t &|&\beta\end{ppmatrix}\right) \]
for any choice of scalars $\lambda_{\alpha,\beta}\in\FF_{q^2}$,
\[ \zeta(A_t,B_t):=\sum_{\alpha\neq0\neq\beta}\psi^E\left(\begin{ppmatrix} A_t-\lambda_{\alpha,\beta} B_t &|& \alpha-\lambda_{\alpha,\beta} \beta \\
      B_t&|&\beta \end{ppmatrix}\right).
\]
Put $\lambda_{\alpha,\beta} :=\alpha\beta^{-1}$. Then,
\[ \zeta(A_t,B_t):=\sum_{\alpha\neq0\neq\beta}\psi^E\left(\begin{ppmatrix} A_t-\alpha\beta^{-1} B_t &|& 0 \\
      B_t&|&\beta \end{ppmatrix}\right).
\]
Since $\alpha$ is an arbitrary non-zero element, for each value of $\beta$, $\lambda_{\alpha,\beta}:=\alpha\beta^{-1}$ assumes all possible
non-zero values --- so, with a change of variable, we get
\[ \zeta(A_t,B_t):=\sum_{\lambda\neq0\neq\beta}\psi^E\left(\begin{ppmatrix} A_t-\lambda B_t &|& 0 \\
      B_t&|&\beta \end{ppmatrix}\right).
\]
Finally, observe that for $\beta\neq0$,
\[ \psi^E\left(\begin{ppmatrix} A_t-\lambda B_t &|& 0 \\
      B_t&|&\beta \end{ppmatrix}\right)=
   \psi^E\left(\begin{ppmatrix} A_t-\lambda B_t &|& 0 \\
       B_t&|&1 \end{ppmatrix}\right). \]
 Hence
\[ \zeta(A_t,B_t):=(q^2-1)\sum_{\lambda\neq0}\psi^E\left(\begin{ppmatrix} A_t-\lambda B_t &|& 0 \\
      B_t&|&1 \end{ppmatrix}\right).
\]




            Now, since the $(t+1)$-entry in the second row is non-zero, by case (E\ref{E3}) analyzed in Section~\ref{t even},
            \[ \psi^E\left(
                \begin{ppmatrix} A_t-\lambda B_t &|&0 \\
                  B_t      &|&1
                \end{ppmatrix}\right)=\frac{1}{q^2}\theta(A_t-\lambda B_t|0)\theta(B_t|1)
            \]
            and
            \[ \zeta(A_t,B_t)=\frac{q^2-1}{q^2}\theta(B_t|1)\sum_{\lambda\neq0}\theta(A_t-\lambda B_t|0). \]

            On the other hand, by Section~\ref{sec3} case 2 of even prefix,
            \[ \theta(A_t-\lambda B_t|0)=q^2\theta(A_t-\lambda B_t|0|0), \]
            so,
            \[ \frac{1}{q^2}\theta(A_t-\lambda B_t|0)=\begin{cases}
                s & \text{if } \eta_t(A_t-\lambda B_t,A_t-\lambda B_t)=0 \\
                \displaystyle\frac{(q^2)^{(m-t-2)}-s}{q-1} &
                \text{if } \eta_t(A_t-\lambda B_t,A_t-\lambda B_t)\neq 0,
              \end{cases}\]
            where $s:=(q^2-1)\mu_{m-t-2}+1$.
Since $\xi(A_t,B_t)$ is the number of values $\lambda\in\FF_{q^2}$ with $\lambda\neq0$   and $\eta_m(A_t-\lambda B_t|0^{m-t},A_t-\lambda B_t|0^{m-t})=0$, we have
\[\begin{array}{l}
    \zeta(A_t,B_t)=(q^2-1)\theta(B_t|1)(\xi(A_t,B_t)s+
              (q^2-1-\xi(A_t,B_t))\frac{q^{2m-2t-4}-s}{q-1})= \\
    \quad =(q^2-1)q^{2m-2t-3}(\xi(A_t,B_t)s+
              (q^2-1-\xi(A_t,B_t))\frac{q^{2m-2t-4}-s}{q-1}).
\end{array} \]
          \end{proof}
          We now compute $\xi(A_t,B_t)$.
          \begin{lemma}
            \label{lxi}
             We have
              \[ \xi(A_t,B_t):=\begin{cases}
                  q^2-1 & \text{if } \eta_t(A_t,A_t)=\eta_t(B_t,B_t)=\eta_t(A_t,B_t)=0 \\
                  q-1 & \text{if } \eta_t(A_t,A_t)=\eta_t(B_t,B_t)=0\neq\eta_t(A_t,B_t) \\
                  0 & \text{if } \eta_t(A_t,B_t)=0=\eta_t(A_t,A_t)\eta_t(B_t,B_t) \\
                  q+1 & \text{if } \eta_t(A_t,B_t)=0, \eta_t(A_t,A_t)\eta_t(B_t,B_t)\neq 0 \\
                  q & \text{if } \eta_t(A_t,B_t)\neq0, \eta_t(A_t,A_t)\eta_t(B_t,B_t)=0 \\
                  1 & \text{if Equation~\eqref{e0f} satisfied}\\
                  q+1 & \text{otherwise} \\
                \end{cases} \]
              where for $\eta_t(B_t,B_t)\neq0$, let
            \[ \delta:=\frac{\eta_t(A_t,B_t)}{\eta_t{(B_t,B_t)}}^q \]
            and consider
                \begin{equation}
                  \label{e0f}
               \delta^{q+1}\eta_t(B_t,B_t)-\delta\eta_t(A_t,B_t)-\delta^{q}\eta_t(B_t,A_t)+\eta_t(A_t,A_t)=0.
             \end{equation}
              \end{lemma}
              \begin{proof}
                The value $\xi(A_t,B_t)$ is precisely the number of
                intersections different from $\{A_t,B_t\}$, of the line $\ell_{AB}:=\langle A_t,B_t\rangle$ with
                a Hermitian variety $\cH_t$ defined by the form $\eta_t$.
                We distinguish several cases:
                \begin{enumerate}
                \item 
                  $\eta_t(A_t,A_t)=\eta_t(B_t,B_t)=0$. Then either the line $\ell_{AB}$ is totally isotropic,
                i.e. $\eta_t(A_t,B_t)=0$; hence
                there are $q^2-1$ points on $\ell_{AB}$ different from
                $A_t$ and $B_t$ and $\xi(A_t,B_t):=q^2-1$ or the line $\ell_{AB}$ meets $\cH_{t}$ in a Baer subline, hence $\xi(A_t,B_t)=q-1.$

            \item $\eta_t(A_t,B_t)=0$ and ($\eta_t(A_t,A_t)\neq0$ or $\eta_t(B_t,B_t)\neq0$). Then
              \begin{enumerate}
              \item $\eta_t(A_t,A_t)=0\neq\eta_t(B_t,B_t)$, then the line $\ell_{AB}$
                is tangent
                to $\cH_{t}$ in $A_t$ but not totally isotropic;
                so it meets $\cH_{t}$ only
                in $A_t$ and, consequently,
                \[ \xi(A_t,B_t)=0; \]
                the case $\eta_t(A_t,A_t)\neq0=\eta_t(B_t,B_t)$ is entirely analogous;
              \item $\eta_t(A_t,A_t)\neq0\neq\eta_t(B_t,B_t)$; then $A_t,B_t\not\in\cH_t$ and $B_t$ is in the polar hyperplane of
                $A_t$. In this case the line $\ell_{AB}$ meets $\cH_{t}$ in a Baer subline, hence there are $(q+1)$ points,
                all different from $A_t$ and $B_t$, so
                \[ \xi(A_t,B_t)=q+1. \]
              \end{enumerate}
            \item $\eta_t(A_t,B_t)\neq0$. Then
              \begin{enumerate}
              \item if $\eta_t(A_t,A_t)=0$ or $\eta_t(B_t,B_t)=0$, then
                the line $\ell_{AB}$ meets $\cH_{t}$ in $q+1$
                points, one of them being $A_t$ (or $B_t$);
                so
                \[ \xi(A_t,B_t)=q. \]
              \item If $\eta_t(A_t,A_t)\neq0\neq\eta_t(B_t,B_t)$, then
                observe that since $A_t,B_t\not\in\cH_{t}$ the line $\ell_{AB}$
                is tangent
                to $\cH_{t}$ if, and only if
                the equation
                \begin{equation}
                  \label{e00} \eta_t(A_t-\lambda B_t,A_t-\lambda B_t)=0
                \end{equation}
                admits a solution with multiplicity $q+1>1$.
                Expanding Equation~\eqref{e00}, we have
                \[ \lambda^{q+1}\eta_t(B_t,B_t)-\lambda^q\eta_t(B_t,A_t)-\lambda\eta_t(A_t,B_t)+\eta_t(A_t,A_t)=0, \]
                and its derivative in $\lambda$ is
                \begin{equation}
                  \label{e0l}
                  \lambda^q\eta_t(B_t,B_t)-\eta_t(A_t,B_t)=0.
                \end{equation}
                So, let
                \[ \delta:=\frac{\eta_t(A_t,B_t)}{\eta_t({B_t,B_t})}^q. \]
                We have that $\ell_{AB}$ is tangent to $\cH_{t}$ if and only if
                \eqref{e0f} is satisfied.
                This gives the last two possible values for $\xi(A_t,B_t)$.
              \end{enumerate}
            \end{enumerate}
          \end{proof}
          \begin{corollary}
            \label{czeta}
            The complexity of the function $\zeta(A_t,B_t)$ is
            $O(t)\leq O(m)$.
          \end{corollary}
          \begin{proof}
            In order to determine $\zeta(A_t,B_t)$ we need to
            evaluate the point enumerator $\theta(B_t|1)$, which
            has complexity $O(m)$, and the function $\xi(A_t,B_t)$.
            It is straightforward to see from Lemma~\ref{lxi} that
            evaluating $\xi(A_t,B_t)$ requires to compute three
            sesquilinear products, each of them involving
            $t$ products and $t$ conjugations. So the complexity of
            $\xi(A_t,B_t)$ is also $O(t)\leq O(m)$.
          \end{proof}
          Observe that
          \[ \psi^E\left(\begin{ppmatrix}A_t&|&\alpha\\
                B_t&|&0\end{ppmatrix}\right)=
            \psi^E\left(\begin{ppmatrix}A_t&|&1\\
                B_t&|&0\end{ppmatrix}\right) \]
          for any $\alpha\neq 0$; so
          \[            \sum_{\alpha\neq0}\psi^E\left(\begin{ppmatrix}A_t&|&\alpha\\               B_t&|&0\end{ppmatrix}\right)=
            (q^2-1)\psi^E\left(\begin{ppmatrix}A_t&|&1\\
                B_t&|&0\end{ppmatrix}\right) \]
          On the other hand,
            \[ \psi^E\left(\begin{ppmatrix}A_t&|&0\\
                  B_t&|&\beta\end{ppmatrix}\right)=
              \psi^E\left(\begin{ppmatrix}A_t&|&0\\
                    B_t&|&1\end{ppmatrix}\right) \]
                for any $\beta\neq 0$; so
                \[
                    \sum_{\beta\neq0}\psi^E\left(\begin{ppmatrix}A_t&|&0\\
                        B_t&|&\beta\end{ppmatrix}\right)=
                    (q^2-1)\psi^E\left(\begin{ppmatrix}A_t&|&0\\
                        B_t&|&1\end{ppmatrix}\right) \]
                  Finally,
                  \[     \psi^E\left(\begin{ppmatrix}A_t&|&0\\
                        B_t&|&0\end{ppmatrix}\right)=\sum_{\alpha,\beta}
                    \psi^O\left(\begin{ppmatrix}A_t&|&0&|&\alpha\\
                        B_t&|&0&|&\beta\end{ppmatrix}\right);
                  \]
                  on the other hand
                  \[ \psi^O\left(\begin{ppmatrix}A_t&|&0&|&\alpha\\
                        B_t&|&0&|&\beta\end{ppmatrix}\right)=
                    \psi^O\left(\begin{ppmatrix}A_t&|&0&|&0\\
                        B_t&|&0&|&0\end{ppmatrix}\right), \]
                  so
                  \[     \psi^E\left(\begin{ppmatrix}A_t&|&0\\
                        B_t&|&0\end{ppmatrix}\right)=q^4
                  \psi^O\left(\begin{ppmatrix}A_t&|&0&|&0\\
                      B_t&|&0&|&0\end{ppmatrix}\right). \]
                It follows that
          \[ \boxed{
              \begin{array}{ll}
                \psi^O(D_t)&:= \zeta(A_t,B_t)+
                             (q^2-1)\bigg(\psi^E\left(\begin{ppmatrix}A_t&|&0\\
                  B_t&|&1\end{ppmatrix}\right)+
              \psi^E\left(\begin{ppmatrix}A_t&|&1\\
                  B_t&|&0\end{ppmatrix}\right)\bigg)+ \\ & q^4
            \psi^O\left(\begin{ppmatrix}A_t&|&0&|&0\\
                B_t&|&0&|&0\end{ppmatrix}\right).
              \end{array}}\]

                    \end{enumerate}

\section{Proof of the Main Theorem}
\label{sec-compl}
We now estimate the complexity of the line enumerator
given by the function $\iota$ described by Equation~\eqref{iota} in
Section~\ref{sec2.1} and complete the proof of our main theorem.
By definition,
in order to evaluate $\iota(\ell)$
for a given line $\ell$ we need to compute the value of the function
$\psi$ defined in Section~\ref{sec4} for at most $q^4m$ different inputs
(recall that the alphabet $\fA$ in this case is $\FF_{q^2}^2$).
So we can state that the complexity of $\iota$ is at most $q^4m$ times
the complexity of $\psi$.
We can now proceed to analyze the function $\psi$ step by step.
We refer to the labels in Section~\ref{t even} and \ref{t odd}, as well
as to Table~\ref{t-compl} for the various cases.

\begin{itemize}
\item
  In case (E\ref{E0}), i.e. $t=0$, as well as in cases
  (E\ref{E1}), (E\ref{E2}) and
  (O\ref{O2}), the number $\psi(D_t)$ depends only
  on the length $t$. So, the complexity is $O(1)$.
\item
  In case (O\ref{O1}), i.e. $t=m$, $\psi(D_m)$ is
  either $1$ or $0$ according as $D_m=\begin{ppmatrix} A\\B\end{ppmatrix}$ represents a totally isotropic line $\ell=\langle A,B\rangle$ or not.
  To check whether this
  is the case we need to evaluate $\eta_m(A,A)$, $\eta_m(B,B)$ and $\eta_m(A,B)$.
  This requires $3m$ products as well as $3m$ conjugations; so the overall
  complexity of this case is $O(m)$.
  \item
  In cases (E\ref{E3}) and (E\ref{E4}) we need to evaluate the number of totally isotropic
  points   whose normalized coordinates begin with either $A$ or $B$, i.e.
  the functions $\theta(A)$ and $\theta(B)$.
  By Lemma~\ref{complexity theta}, the complexity of $\theta$ is $O(m)$; consequently
  the complexity of this case is also $O(m)$.
  \item
    The value of $\psi^E(D_t)$ in cases (E\ref{E5}) and (E\ref{E6})
    is a function of $\psi^O(D'_{t+1})$ where
    $D'_{t+1}$ is as in cases
     (O\ref{O3}) or  (O\ref{O4}).
     More in detail,
     the complexity of case (E\ref{E5}) is the same as the complexity
     of case (O\ref{O4}), while the complexity of case (E\ref{E6}) is
     the same as the sum of the complexities of cases (O\ref{O3}) and
     (O\ref{O4}).
\item
  The computation of $\psi^O(D_t)$ in both cases (O\ref{O3}) and (O\ref{O4})
  requires to compute the value of $\psi^E(D_{t+1})$ in cases of
  the form
  (E\ref{E2}),
  (E\ref{E3}) or (E\ref{E4})
  and then evaluate functions of the form $\psi^O(D_{t+2})$ with suitable
  prefixes of length $t+2$.
  We have already seen that the complexity of cases
  (E\ref{E2}),
  (E\ref{E3}) or
  (E\ref{E4})
  is at most $O(m)$.

  If $t+2=m$, then $\psi^O(D_{t+2})$ is
  of type (O\ref{O1}) and has also complexity $O(m)$, so, for $t=m-2$
  the overall complexity of both (O\ref{O3}) and (O\ref{O4}) is
  $O(m)$.

  Suppose now $t=m-2i$ with $1\leq i\leq \lfloor m/2\rfloor$.
  The complexity of the
  function $\zeta(A_t,B_t)$ occurring in case (O\ref{O4})
  is $O(m)$, see Corollary~\ref{czeta}.

  In any case, the complexity of $\psi^O(D_{m-2i})$
  is the sum of $O(m)$ and  the complexity of $\psi^O(D_{m-2i+2})$.
  In turn, the complexity of $\psi^O(D_{m-2i+2})$ is $O(m)$ plus
  the complexity of $\psi^O(D_{m-2i+4})$. After $i$  steps,
  the complexity of $\psi^O(D_{m-2i})$ is $i$ times the
  complexity of $\psi^O(D_m)$. Since $i=O(m)$, we see that
  the total complexity for cases (O\ref{O3}) and (O\ref{O4})
  is at most $O(m^2)$.
  \end{itemize}
  The above argument shows that the maximum complexity of the
  function $\psi$ is
  $O(m^2)$. So, the complexity of $\iota$ is $O(q^4m^3)$.
  \par
  \rightline{$\qed$}

\begin{table}
  \newcolumntype{C}{>$c<$}
  \caption{Enumerator for Hermitian Line Grassmannians}
  \label{t-compl}
  \begin{small}
    \[
      \begin{array}{c|C|C|c|c}
      D_t=\begin{ppmatrix} A_t\\ B_t\end{ppmatrix} & {t} & \text{Case} & \psi(S) & \text{Complexity} \\ \hline
\emptyset & $t=0$ & (E\ref{E0})  &
\begin{array}[c]{c}
\\
N_m \\
\\
\end{array}  & O(1) \\ \hline
            \begin{ppmatrix}
       a_1 & a_2 & \cdots  & a_{m} \\
       b_1 & b_2 & \cdots   & b_{m} \\
     \end{ppmatrix} & $t=m$ & (O\ref{O1}) & \begin{cases}
       1 & \text{if } \begin{array}[t]{l} \eta_m(A_m,B_m)=\eta_m(A_m,A_m)\\
                        \quad=\eta_m(B_m,B_m)=0 \\
                        \end{array} \\
       0 & \text{otherwise} \end{cases}
       & O(m) \\ \hline

     \begin{ppmatrix}
       0 & 0 & \cdots & 0 & 0 \\
       0 & 0 & \cdots  &0 & 0 \\
     \end{ppmatrix} & \mbox{Even} & (E\ref{E1}) &
                                                  \mu_{m-t-1}+q^4N_{m-t-1} & O(1) \\
        \hline
        \begin{array}{c}
        \begin{ppmatrix}
       a_1 & a_2 & \cdots & a_{t-1} & a_t \\
       0 & 0 & \cdots & 0 & 0 \\
     \end{ppmatrix} \\
          a_t\neq 0
        \end{array}
                   & \mbox{Even} & (E\ref{E2}) & q^{2m-2t-3}\mu_{m-t-1}
                                                                                 & O(1) \\ \hline

\begin{array}{c}
     \begin{ppmatrix}
       a_1 & a_2 & \cdots & a_{t-1} & 0 \\
       b_1  & b_2  & \cdots & b_{t-1}  & b_t
     \end{ppmatrix} \\
           b_t\neq0\end{array}&
                     \mbox{Even} & (E\ref{E3}) & \frac{1}{q^2}\theta(A_t)\theta(B_t) &
                                                                       O(m) \\
\hline
        \begin{array}{c}
     \begin{ppmatrix}
       a_1 & a_2 & \cdots & a_{t-1} & a_t \\
       b_1  & b_2  & \cdots & b_{t-1}  & 0
     \end{ppmatrix} \\
  a_t\neq 0,  B_t\neq\mathbf{0}
  \end{array}
            &
                     \mbox{Even} & (E\ref{E4}) & \frac{1}{q^2}\theta(A_t)\theta(B_t) &
                                                                       O(m) \\
\hline
        \begin{array}{c}   \begin{ppmatrix}
       a_1 & a_2 & \cdots & a_{t-1} & 0 \\
       b_1  & b_2  & \cdots & b_{t-1}  & 0
     \end{ppmatrix} \\
    A_t\neq\mathbf{0}\neq B_t
    \end{array}
                                                   & \mbox{Even} & (E\ref{E5}) & q^4\psi^O\begin{ppmatrix} A_t&|&0 \\ B_t&|&0 \end{ppmatrix} & O(m^2) \\
        \hline
        \begin{array}{c}
     \begin{ppmatrix}
       a_1 & a_2 & \cdots & a_{t-1} & 0 \\
       0  & 0  & \cdots & 0  & 0
     \end{ppmatrix} \\
          A_t\neq\mathbf{0}
        \end{array}
                                                   & \mbox{Even} & (E\ref{E6}) & q^4\psi^O\begin{ppmatrix} A_t&|&0 \\ \mathbf{0}&|&0 \end{ppmatrix}+\psi^O\begin{ppmatrix} A_t&|&0\\ \mathbf{0}&|&1\end{ppmatrix} & O(m^2) \\
        \hline
     \begin{ppmatrix}
       0 & 0 & \cdots & 0 & 0 \\
       0 & 0 & \cdots  &0 & 0 \\
     \end{ppmatrix} & \mbox{Odd} & (O\ref{O2}) & N_{m-t} & O(1) \\
        \hline
        \begin{array}{c}
     \begin{ppmatrix}
       a_1 & a_2 & \cdots  & a_{t} \\
       0      &   0      & \cdots         &  0
     \end{ppmatrix} \\ A_t\neq\mathbf{0}\end{array}
                      & \mbox{Odd} & (O\ref{O3}) &
                                    \begin{array}{l}
                                      (q^2-1)\psi^E\begin{ppmatrix} A_t&|&1\\\mathbf{0}&|&0\end{ppmatrix}+\psi^E\begin{ppmatrix}A_t&|&0\\\mathbf{0}&|&1\end{ppmatrix}+\\
                                      q^2\psi^O\begin{ppmatrix}
                                        A_t&|&0&|&0 \\
                                        \mathbf{0}&|&0&|&0 \\
                                      \end{ppmatrix}+\psi^O
                                    \begin{ppmatrix}
                                        A_t&|&0&|&0 \\
                                        \mathbf{0}&|&0&|&1 \\
                                      \end{ppmatrix}
                                    \end{array} & O(m^2) \\
        \hline
\begin{array}{c}
     \begin{ppmatrix}
       a_1 & a_2 & \cdots  & a_{t} \\
       b_1      & b_2 & \cdots        &  b_t
     \end{ppmatrix} \\
  A_t\neq\mathbf{0}\neq B_t
           \end{array}& \mbox{Odd} & (O\ref{O4}) &
                                    \begin{array}{l}
                                      \zeta(A_t,B_t)+
                                      (q^2-1)\bigg(\psi^E\begin{ppmatrix} A_t&|&1\\B_t&|&0\end{ppmatrix}+\\ \psi^E\begin{ppmatrix}A_t&|&0\\B_t&|&1\end{ppmatrix}\bigg)+
                                      q^4\psi^O\begin{ppmatrix}
                                        A_t&|&0&|&0 \\
                                        B_t&|&0&|&0 \\
                                      \end{ppmatrix}
                                    \end{array} & O(m^2) \\
\hline

   \end{array}
 \]
\centerline{The function $\zeta(A_t,B_t)$ is computed in Lemma~\ref{lemma3} and Lemma~\ref{lxi}.}
\end{small}

\end{table}



\section{Applications}\label{sec6}
\subsection{Encoding and decoding}
Assume $m$ odd and
let $B:=(e_1,\ldots,e_m)$ a  basis of $V(m,q^2)$ such
that the sesquilinear form $\eta$ has Equation~\eqref{eta}.
Our arguments apply also when $m$ is even, and $V(m,q^2)$ is
 the hyperplane of equation $x_{1}=0$ embedded in $V(m+1,q^2)$.
Let also $K={m\choose 2}$ and suppose
$\mathbf{w}=(w_1,\ldots,w_K)$ to be a message.
Using the enumerator $\iota$, we define the $i$-th component $c_i$ of a
codeword $\mathbf{c}=(c_1,\ldots,c_N)$ representing $\mathbf{w}$ as
\[ c_i:=\omega_{\mathbf{w}}(A_{i-1},B_{i-1}):=A_{i-1} W B_{i-1}^T \]
where $\begin{ppmatrix} A_{i-1}\\ B_{i-1}\end{ppmatrix}=\iota^{-1}(i-1)$
represents the $i$--th element $\ell=\langle A_{i-1},B_{i-1}\rangle$
of the polar Grassmannian $\cH_{n,2}$
and $W=W_0-W_0^T$ is the antisymmetric matrix obtained from
\[ W_0=\begin{pmatrix}
    0 & w_1 & w_2  & \ldots & w_{m-1} \\
    &  0  & w_{m} & \ldots & w_{2m-3} \\
    & & \ddots &  &  \vdots \\
    & &        &  0 &  w_K \\
    & & & &  0 \\
  \end{pmatrix} \]
representing the alternating bilinear form $\omega_{\mathbf{w}}$ induced
by $\mathbf{w}$ with respect to the basis $B$. Note that                        $W$ is precisely the Gram matrix of the alternating bilinear form $\omega_{\mathbf{w}}$.

For $1\leq i<j\leq m$, denote by ${\mathfrak w}_{ij}$ the entry in $W_0$ in position $(i,j)$.
Then
\[ {\mathfrak w}_{ij}:=w_{\frac{(i-1)(2n-i)}{2}+j-i}. \]
Denote by $\varphi:\FF_{q^2}^K\to\FF_{q^2}^N$ the function mapping a message
$\mathbf w$ to the corresponding codeword $\mathbf c$.
\begin{lemma}
  \label{llin}
  The function $\varphi:\FF_{q^2}^K\to\FF_{q^2}^N$ is linear.
\end{lemma}
\begin{proof}
Let ${\mathbf w}$ and ${\mathbf w}'$ be two messages and
$\mathbf c$ and ${\mathbf c}'$ be the corresponding codewords.
Let also $\mathbf{c''}=\alpha\mathbf{c}+\beta\mathbf{c}'$ and
$\mathbf{w}''=\alpha\mathbf{w}+\beta\mathbf{w}'$ for $\alpha,\beta\in\FF_{q^2}$.
Then
\[ c_i''=\alpha c_i+\beta c_i'=\alpha A_{i-1} W B_{i-1}^T+
  \beta A_{i-1} W'B_{i-1}^T=A_{i-1}(\alpha W+\beta W')B_{i-1}^T, \]
so $\varphi(\alpha\mathbf{w}+\beta\mathbf{w}')=
\alpha\varphi(\mathbf{w})+\beta\varphi(\mathbf{w}')$ and $\varphi$ is
linear.
\end{proof}

Given a codeword ${\mathbf c}=(c_i)_{i=1}^N$ we now show how to
uniquely extract the entries
${\mathfrak w}_{ij}$ with $i<j$ of $W_0$
from $\mathbf{c}$ with
constant complexity $O(1)$; clearly, this is equivalent to determine the message $\mathbf{w}=(w_i)_{i=1}^K$.

\begin{theorem}
  \label{Decode}
  Suppose $m$ to be odd.
  Let $\mathbf{c}$ be a codeword and $W=({\fmm}_{ij})_{1\leq i,j\leq 2n+1}$ be the antisymmetric matrix associated
  with the message $\mathbf{w}$ mapped to $\mathbf{c}$ using the
  function $\varphi$.
  Suppose that the pair $(i,j)$ with $1\leq i<j\leq m$ is in one of the following types:
\begin{enumerate}[Type I:]
\item ($i\geq 2$ even and $j\geq i+2$)
  or ($i$ odd and $j\geq i+1$);
\item $i\geq 2$ even and $j=i+1$;
\item $i=1$ and $j>i$.
\end{enumerate}
Then the following holds:
\begin{itemize}
\item  If $(i,j)$ is of Type I then ${\fmm}_{ij}=c_{\iota(\ell_{i,j})+1}$ where
  $\ell_{i,j}:=\langle e_i,e_j\rangle$.
\item If $(i,j)$ is of Type II then ${\fmm}_{ij}$  can be obtained by solving a system of $2$ linear equations in $2$ unknowns.
\item If $(i,j)$ is of Type III then ${\fmm}_{ij}$  can be obtained by solving a linear equation.
\end{itemize}
\end{theorem}
\begin{proof}
 If $(i,j)$ is of Type I then the line $\ell_{i,j}:=\langle e_i,e_j\rangle$ is totally
 isotropic for the Hermitian form $\eta$; furthermore
 $\omega_{\mathbf w}(e_i,e_j)={\fmm}_{ij}$
  and we are done.

  When $(i,j)$ is of Type II,
let $\sigma$ be an element of $\FF_{q^2}\setminus\FF_q$ such that
$\sigma^q+\alpha\sigma+\beta=0$ with $\alpha,\beta\in\FF_q\setminus\{0\}$.
Consider two lines
$\ell:=\langle e_i+(\alpha\sigma+\beta) e_{i+3}, \sigma e_{i+1}+e_{i+2}\rangle$ and
$\ell^q:=\langle e_i+(\alpha\sigma^q+\beta)e_{i+3}, \sigma^qe_{i+1}+e_{i+2}\rangle$.
It is straightforward to see that both lines are totally isotropic, so they correspond
to two components in the codeword ${\mathbf c}$;
call them respectively $c_x:=c_{\iota(\ell)+1}$ and $c_y:=c_{\iota(\ell^q)+1}$.
Then we have
\begin{equation}\label{typII} \begin{cases}
  \sigma {\fmm}_{i,i+1}+{\fmm}_{i,i+2}-\sigma(\alpha\sigma+\beta){\fmm}_{i+1,i+3}-(\alpha\sigma+\beta){\fmm}_{i+2,i+3}=c_x \\
\sigma^q
{\fmm}_{i,i+1}+{\fmm}_{i,i+2}-\sigma^q(\alpha\sigma^q+\beta){\fmm}_{i+1,i+3}-(\alpha\sigma^q+\beta)
\mathfrak{m}_{i+2,i+3}=c_y.
\end{cases} \end{equation}
The entries ${\fmm}_{i+1,i+3}$ and ${\fmm}_{i,i+2}$ correspond to indexes of Type I; thus they  can be read off $\mathbf{c}$ directly.
We are left with a linear system of two equations in the unknowns
$\fmm_{i,i+1}$ and $\fmm_{i+2,i+3}$.
Since
\[ \det\begin{pmatrix} \sigma & -(\alpha\sigma+\beta) \\
    \sigma^q & -(\alpha\sigma^q+\beta)
  \end{pmatrix}=\beta(\sigma^q-\sigma)\neq0,
\]
this linear system admits a unique solution and can be solved with
complexity $O(1)$.

Suppose $(i,j)=(1,j)$ is of Type III.
If $j>3$, we consider the line $\ell=\langle e_1-e_2+e_3,e_j\rangle$.
A straightforward computation shows that the corresponding entry $c_z:=c_{\iota(\ell)+1}$ is
\[ \fmm_{1j}-\fmm_{2j}+\fmm_{3j}=c_z \]
and both $(2,j)$ and $(3,j)$ are of Type I; thus we just have to solve this
equation.
As for the remaining
coefficients ${\fmm}_{12}$ and ${\fmm}_{13}$, we use the entries
corresponding to $\ell^{12}=\langle e_1-e_4+e_5,e_2\rangle$ and
$\ell^{13}=\langle e_1-e_4+e_5,e_3\rangle.$
\end{proof}

\begin{corollary}
  \label{dcore}
  Suppose $m$ to be even.
  Let $\mathbf{c}$ be a codeword and $W=({\fmm}_{ij})_{1\leq i,j\leq 2n+1}$ be the antisymmetric matrix associated
  with the message $\mathbf{w}$ mapped to $\mathbf{c}$ using the
  encoding $\varphi$. Suppose that the pair $(i,j)$ with $1\leq i<j\leq m$ is in one of the following types:
\begin{enumerate}[Type I:]
\item ($i\geq 1$ odd and $j\geq i+2$)
  or ($i$ even and $j\geq i+1$);
\item $i\geq 2$ odd and $j=i+1$;
\end{enumerate}
Then the following holds:
\begin{itemize}
\item If $(i,j)$ is of Type I then ${\fmm}_{ij}=c_{\iota(\ell_{i,j})+1}$ where
$\ell_{i,j}:=\langle e_i,e_j\rangle$.
\item If $(i,j)$ is of Type II then ${\fmm}_{ij}$  can be obtained by solving a system of $2$ linear equations in $2$ unknowns.
\end{itemize}
\end{corollary}
\begin{proof}
  For $m$ even, we can regard the Hermitian polar space $\cH_m$ as
  the hyperplane section of $\cH_{m+1}$ with respect to the hyperplane
  $x_1=0$. By renumbering the indexes, we can write the form inducing
  $\cH_n$ as $\eta(x,y)=x_1y_2^q+\cdots$ instead of
  $x_1y_1^q+x_2y_3^q+\cdots, x_1=y_1=0$. The corollary now
  follows from the previous theorem.

\end{proof}
\begin{corollary}
  \label{cinj}
  The map $\varphi:\FF_{q^2}^K\to\FF_{q^2}^N$ is injective.
\end{corollary}
\begin{proof}
  Suppose $\varphi(\mathbf{w})=\mathbf{0}$; by the proof of Theorem~\ref{Decode}, all indexes $\fmm_{ij}$ of Type I must be $0$.
  Also, for indexes of Type II System~\eqref{typII} becomes
  \[ \begin{cases}
      \sigma\fmm_{i,i+1}-(\alpha\sigma+\beta)\fmm_{i+2,i+3}=0 \\
      \sigma^q\fmm_{i,i+1}-(\alpha\sigma^q+\beta)\fmm_{i+2,i+3}=0 \\
    \end{cases} \]
  whose only solution is $\fmm_{i,i+1}=\fmm_{i+2,i+3}=0$.
  Finally, entries of Type III must satisfy
  $\fmm_{1j}=0$ for all $j$.

  The case $m$ even is entirely analogous and can be proven using
  Corollary~\ref{dcore}.
\end{proof}
 By Lemma~\ref{llin} and Corollary~\ref{cinj}, $\varphi$ is
 a linear encoding.
\subsection{Error correction}
First of all, observe that in order to recover the original message
sent $\mathbf{w}=(w_1,\ldots,w_K)$ it is enough to guarantee that the
entries of a received vector $\mathbf{r}\in\FF_{q^2}^N$ needed to obtain
the elements $\fmm_{ij}$ of Theorem~\ref{Decode} and Corollary~\ref{dcore}, are
correct. This could be implemented using standard techniques from coding
theory, e.g. syndrome decoding, see \cite[Chapter 1]{MS}, but such an approach would be very inefficient
in the case of (polar)
Grassmann codes, since the parity check matrix is  huge.

A different, more viable, approach is what we proposed in~\cite{IL17}
for line polar Grassmann codes of either orthogonal or symplectic type and we
here extend to the Hermitian case.
Suppose $r_x$ is an entry in the received vector $\mathbf{r}$ which we want
to insure to be correct. So, we take the line $\ell=\langle A_{x-1},B_{x-1}
\rangle$ with index $\iota(\ell)=x-1$ and consider the pencil
$\Pi_{\ell}$ of all the totally singular
planes passing through $\ell$.

For each $\pi\in\Pi_{\ell}$ choose $3$ non-concurrent lines $r,s,t$ of $\pi$ different from $\ell$. Observe that the values of $r_{\iota(r)+1}$
$r_{\iota(s)+1}$ and
$r_{\iota(t)+1}$ are
sufficient to reconstruct the restriction of the
alternating form $\omega_{\mathbf w}$
to $\pi$, say $\omega^{\pi}$. If
$\omega^{\pi}(A_{x-1},B_{x-1})=r_x$ for a sufficient number
of planes, then we can assume that the received value $r_x$ is
correct. Otherwise, we replace $r_x$ with the majority of
the values $\omega^{\pi}(A_{x-1},B_{x-1})$ assumes as $\pi$ varies in $\Pi_{\ell}$.
We leave to a future work a detailed analysis of the performance of
this error correcting algorithm.

\section*{Acknowledgments}
Both authors are affiliated with GNSAGA of INdAM (Italy) whose support they
acknowledge.


\vskip.2cm
\noindent
\begin{minipage}[t]{\textwidth}
Authors' addresses:
\vskip.2cm\noindent\nobreak
\centerline{
\begin{minipage}[t]{7cm}
Ilaria Cardinali\\
Department of Information Engineering and Mathematics\\University of Siena\\
Via Roma 56, I-53100, Siena, Italy\\
ilaria.cardinali@unisi.it\\
\end{minipage}\hfill
\begin{minipage}[t]{7cm}
Luca Giuzzi\\
D.I.C.A.T.A.M. \\ Section of Mathematics \\
University of Brescia\\
Via Branze 43, I-25123, Brescia, Italy \\
luca.giuzzi@unibs.it
\end{minipage}}
\end{minipage}

\end{document}